\documentclass[11pt,a4paper]{article}

\usepackage{amsmath}
\usepackage{amssymb}
\usepackage{stmaryrd}
\makeatletter

\usepackage[margin=1in]{geometry}
\usepackage{hyperref}
\usepackage{mathtools}
\usepackage{amsthm}
\usepackage{upquote}
\usepackage{amsfonts}
\usepackage{amssymb}\usepackage{amsmath}
\usepackage{amsfonts}
\usepackage{amssymb}
\usepackage{bbm}
\usepackage{enumitem}
\usepackage{tikz}
\usetikzlibrary{decorations.markings}
\usepackage{tikz-cd}
\usepackage{authblk}
\usetikzlibrary{babel}

\newtheorem{theorem}{Theorem}[section]

\newtheorem{proposition}[theorem]{Proposition}

\newtheorem{example}{Example}[section]
\theoremstyle{remark}

\AtEndDocument{
\textsc{Mathematical and Statistical Sciences, University of Alberta, Edmonton, Canada} \\
\textsc{Email:} \texttt{primozic@ualberta.ca}}

\begin{document}
\title{On universal stably free modules in positive characteristic}
\author{Eric Primozic}

\maketitle
\setcounter{secnumdepth}{1}

\renewcommand{\thefootnote}{\fnsymbol{footnote}} 
\renewcommand{\thefootnote}{\arabic{footnote}} 

\begin{abstract}
We study the mod $p$ motivic cohomology of homogeneous varieties such as $GL_{n}/GL_{r}$ or $Sp_{2n}/Sp_{2n-2}$ along with the action of the Steenrod operations, without restrictions on the characteristic of the base field. In particular, we prove that certain quotient maps do not admit sections.
\end{abstract}
\section*{Introduction}
Fix a field $k$ and a $k$-algebra $A.$ An $A$-module $P$ is said to be stably free if $$P \oplus A^{n-q} \cong A^{n}$$ for some integers $0 \leq q \leq n.$ There is a universal stably free $A_{n,q}$-module $P_{n,q}$ over a $k$-algebra $A_{n,q}$ where $$\textup{Spec}(A_{n,q}) \cong GL_{n}/GL_{q}$$ \cite[\S 1]{Ray}. The module $P_{n,q}$ is free if and only if the quotient map 
\begin{equation} \label{firstquotient}GL_{n} \to GL_{n}/GL_{q}\end{equation} admits a section. For $r \leq q,$ the module $P_{n,q}$ has a free factor of rank $r$ if and only if the map \begin{equation} \label{secondquotient}GL_{n}/GL_{q-r} \to GL_{n}/GL_{q}\end{equation} admits a section. 

\indent Using Steenrod operations on \'etale cohomology with finite coefficients, Raynaud showed in \cite[Th\'eor\`eme 6.1 and Th\'eor\`eme 6.6]{Ray}:
\begin{enumerate}
\item The fibration $$GL_{n} \to GL_{n}/GL_{n-1}$$ does not admit a section except possibly when $\textup{char}(k)=2, n=3$ or $\textup{char}(k)=3, n=4.$
\item The fibrations $$Sp_{2n} \to Sp_{2n}/Sp_{2n-2}, \, \,  SO_{2n+1} \to SO_{2n+1}/SO_{2n-1}$$ do not admit sections except possibly when $\textup{char}(k)=3, n=2$ or $\textup{char}(k)=5, n=3.$
\item The fibration $$GL_{n}/GL_{n-q} \to GL_{n}/GL_{n-1}$$ does not admit a section, except possibly if $n$ is divisible by $$N_{q}(\textup{char}(k))=\prod_{\textup{prime} \, \, p\neq \textup{char}(k)}p^{1+n(p,q)}$$ where $n(p,q)$ is the largest integer $h \geq -1$ such that $p^{h}(p-1) \leq q-1.$
\end{enumerate}
 Raynaud posed the question of whether the restrictions on $\textup{char}(k)$ could be removed in these results \cite[page 21]{Ray}. The goal of this paper is to give a positive answer to Raynaud's question by studying the action of Steenrod operations on motivic cohomology with finite coefficients. 
 
 \indent After we compute the relevant motivic cohomology groups and determine the action of the Steenrod operations, our arguments are similar to what can be done in topology \cite{BorSer}. The main new tools we have to play with are the Steenrod operations from \cite{Pri1}, which act on the mod $p$ motivic cohomology of smooth schemes over $k$ with $\textup{char}(k)=p>0.$ 
 
 \indent Previously, Williams showed how to use motivic cohomology and Steenrod operations to obtain the results from \cite{Ray}, assuming that the characteristic of the base field and the relevant coefficient field are different. We also note that Mohan Kumar and Nori previously showed that the fibration \ref{firstquotient} doesn't admit a section for $n \geq 3,$ without any restrictions on $\textup{char}(k)$ (they proved a more general statement on when projective modules defined by certain unimodular rows are free) \cite[Theorem 17.1]{Swa}. Our other results seem to be new.

\section*{Acknowledgments} I thank Ben Williams for suggesting to me the problem considered in this paper.
\section{Setup} Fix a field $k$ and let $\textup{Sm}(k)$ denote the category of separated smooth schemes of finite type over $k.$ Let $H(k), H_{\bullet}(k)$ denote the unpointed and pointed motivic homotopy categories respectively of spaces over $k$ \cite{MorVoe}. Let $SH(k)$ denote the stable motivic homotopy category of spectra over $k$ \cite{Voe3}. There is an adjunction $$\Sigma^{\infty}_{+}: H(k) \rightleftarrows SH(k): \Omega^{\infty}$$ where $\Sigma^{\infty}_{+}$ denotes the infinite $\mathbb{P}^{1}$-suspension functor and $\Omega^{\infty}$ is the infinite delooping functor.
\indent For $X \in \textup{Sm}(k)$ and a coefficient ring $A,$ we let $H^{i}(X, A(j))$ denote the motivic cohomology group of $X$ of degree $i$ and weight $j.$ There is an isomorphism between motivic cohomology and higher Chow groups: $$H^{i}(X, \mathbb{Z}(j)) \cong CH^{j}(X, 2j-i).$$ In particular, $H^{i}(X, \mathbb{Z}(j))=0$ for $i>2j.$ For a coefficient ring $A,$ we let $HA \in SH(k)$ denote the motivic Eilenberg-MacLane spectrum. There are Eilenberg-MacLane spaces $K(i,j ,A) =\Omega^{\infty}\Sigma^{i,j}HA \in H_{\bullet}(k)$ for all $i, j \geq 0.$

\section{Higher Chern classes and Steenrod operations}
\indent For a flat affine group scheme $G/k,$ we consider the geometric classifying space $BG \in H_{\bullet}(k).$ Here, $BG$ is defined in the sense of Totaro \cite{TotChow} and Morel-Voevodsky \cite{MorVoe} by approximating $BG$ by certain smooth quotients $U/G.$ For $X \in \textup{Sm}(k)$ such that $G$ acts on $X,$ $H_{G}^{i}(X, \mathbb{Z}(j))$ is similarly defined to be the motivic cohomology $H^{i}((X\times U)/G, \mathbb{Z}(j))$ for a suitable $U.$

\indent Let $S \in H_{\bullet}(k)$ denote the simplicial sphere and let $T \in H_{\bullet}(k)$ denote the Tate sphere. For the infinite general linear group $GL$ and $X \in \textup{Sm}(k),$ there is an isomorphism $$\textup{Hom}_{H_{\bullet}(k)}(\Sigma^{m}_{T}\Sigma^{n}_{S}X_{+}, BGL \times \mathbb{Z}) \cong K_{n-m}(X)$$ for $n, m \geq 0$ \cite[Theorem 3.13]{MorVoe}. From \cite{Pus}, the motivic cohomology of $BGL$ is given by $$H^{*}(BGL, \mathbb{Z}(*)) \cong H^{*}(k, \mathbb{Z}(*))[c_{1}, c_{2}, \ldots ]$$ where $c_{i}$ has bidegree $(2i,i)$ for all $i \in \mathbb{N}.$

\indent These facts give rise to higher Chern classes. Let $X \in \textup{Sm}(k).$ For $m \geq 0$ and $i \in \mathbb{N},$ we get a map 
\[ \small
\begin{tikzcd}
c_{m,i}:K_{m}(X)=\textup{Hom}_{H_{\bullet}(k)}(\Sigma_{S}^{m}X, \mathbb{Z} \times BGL) \arrow[r, "(c_{i})_{*}"] & \textup{Hom}_{H_{\bullet}(k)}(\Sigma_{S}^{m}X, K(2i,i, \mathbb{Z}))=H^{2i-m}(X, \mathbb{Z}(i)).
\end{tikzcd}
\]
\begin{proposition} Let $\alpha \in K_{1}(GL_{n})$ denote the class of the universal matrix. Then $$H^{*}(GL_{n}, \mathbb{Z}(*)) \cong H^{*}(k, \mathbb{Z}(*)) \otimes \bigwedge(c_{1,1}(\alpha), \ldots, c_{1,n}(\alpha))$$ with the relations $c_{1,i}(\alpha)^{2}=\rho c_{1,2i-1}(\alpha)$ holding for all $i$ where $\rho=-1 \in H^{1}(k, \mathbb{Z}(1)).$
\end{proposition}
\begin{proof} This is proved by Pushin \cite[Statement 1]{Pus}.
\end{proof}

\indent Now assume that $\textup{char}(k)=p>0.$ For $n \geq 0,$ there are Steenrod operations $$P^{n}:H\mathbb{F}_{p} \to \Sigma^{2n(p-1), n(p-1)}H\mathbb{F}_{p}$$ \cite{Pri1}. Restricted to mod $p$ Chow groups, the Steenrod operations $P^{n}$ satisfy the following  properties.
\begin{theorem} \label{theorem steenrod}
Let $X \in \textup{Sm}(k).$
\begin{enumerate}
\item $P^{0}$ is the identity.
\item For $n \in \mathbb{N},$ $P^{n}$ is the $p$th power on $CH^{n}(X)/p.$
\item (instability) For $m<n,$ $P^{n}$ is $0$ on  $CH^{m}(X)/p.$
\item (Adem relations) Let $a, b \in \mathbb{N}.$ Restricted to mod $p$ Chow groups, $$P^{a}P^{b}=\sum_{j=0}^{\lfloor \frac{a}{p}\rfloor} (-1)^{a+j}\binom{(p-1)(b-j)-1}{a-pj} P^{a+b-j}P^{j}.$$
\item (Cartan formula) Let $x, y \in CH^{*}(X)/p$ and let $n \in \mathbb{N}.$ Then $$P^{n}(xy)=\sum_{j=0}^{n}P^{j}(x)P^{n-j}(y).$$
\end{enumerate}
\end{theorem}
\indent We can now determine the action of the Steenrod operations on the generators of the mod $p$ motivic cohomology of $GL_{n}.$
\begin{proposition} \label{actionsteenrod}
Let $j \leq n$ and let $i \in \mathbb{N}.$ Then $P^{i}(c_{1,j}(\alpha))=\binom{j-1}{i}c_{1,j+i(p-1)}(\alpha).$
\end{proposition}
\begin{proof}
Essentially, the argument given in \cite{Pus} works here. The argument given in \cite[proof of Theorem 20]{Wil} also works. From Theorem \ref{theorem steenrod}, the Steenrod operations $P^{n}$ have the same action on $H^{*}(BGL, \mathbb{F}_{p}) \cong \mathbb{F}_{p}[c_{1}, c_{2}, \ldots ]$ as their characteristic $0$ counterparts. In particular, $P^{i}(c_{j})=\binom{j-1}{i}c_{j+i(p-1)}$ modulo decomposable elements \cite[Lemma 12 and Lemma 16]{Pus}. The result then follows from the fact that the map
\[
\begin{tikzcd}
\Sigma^{l}_{S}X_{+} \arrow[r, ""] & BGL \arrow[r, "(c_{r}c_{s})_{*}"] &K(2(r+s), r+s, \mathbb{Z})
\end{tikzcd}
\]
is null-homotopic for any $X \in \textup{Sm}(k),$ a map $\Sigma^{l}_{S}X_{+} \to BGL,$  and $l, r, s \in \mathbb{N}$ \cite[Lemma 6]{Pus}.
\end{proof}

\section{Motivic cohomology of homogeneous spaces}
\indent Fix a field $k$ with $\textup{char}(k)=p>0.$ We consider group schemes $GL_{n}, Sp_{2n}, SO_{2n+1}$ defined over $k.$ In this section, we compute the mod $p$ motivic cohomology of homogeneous varieties such as $GL_{n}/GL_{r}$ along with the action of the Steenrod operations. Our main computational tool is the Eilenberg-Moore spectral sequence of Krisha \cite{Kri}. We follow the notation of \cite[Chapter 16]{Tot2}. 

\indent For a bigraded commutative ring $R$ and bigraded $R$-modules $M, N,$ let $\textup{Tor}^{R}_{i,q,j}(M,N)$ denote the $(q,j)$th part of $\textup{Tor}^{R}_{i}(M,N).$ For the definition of torsion and non-torsion primes, see \cite[Section 3]{Pri}. For our purposes, we only need to know that $p$ is a non-torsion prime for $GL_{n}, Sp_{2n}$ for all primes $p$ and $2$ is the only torsion prime for $SO_{2n+1}.$
\begin{theorem} Let $G/k$ be a split reductive group and assume that $p$ is a non-torsion prime for $G.$ Let $X \in \textup{Sm}(k)$ be quasi-projective with a $G$ action. For $j \geq 0,$ there exists a convergent second-quadrant spectral sequence with  
$$E_{2}^{l,q}=\textup{Tor}_{-l,q,j}^{CH^{*}(BG)/p}(\mathbb{F}_{p}, H^{*}_{G}(X, \mathbb{F}_{p}(*)))\Rightarrow H^{l+q}(X, \mathbb{F}_{p}(j)).$$
\end{theorem}

\indent We'll need to know the motivic cohomology of $Sp_{2n}$ and $SO_{2n+1}$ along with the action of the Steenrod operations. For an affine group scheme $G/k,$ we let $G_{(1)}$ denote the first Frobenius kernel of $G.$
\begin{proposition}
\begin{enumerate}
\item The inclusion map $Sp_{2n} \to GL_{2n}$ induces an inclusion $$\bigwedge(c_{1,2}(\alpha), c_{1,4}(\alpha), \cdots, c_{1,2n}(\alpha)) \hookrightarrow H^{*}(Sp_{2n}, \mathbb{F}_{p}(*)) .$$ 
\item Assume that $p>2.$ The inclusion map $SO_{2n+1} \to GL_{2n+1}$ induces an inclusion $$\bigwedge(c_{1,2}(\alpha), c_{1,4}(\alpha), \cdots, c_{1,2n}(\alpha)) \hookrightarrow H^{*}(SO_{2n+1}, \mathbb{F}_{p}(*)) .$$ 
\end{enumerate}
\end{proposition}
\begin{proof}
 After base change, we can assume that the base field $k$ is perfect. From \cite[\S 15]{TotChow}, the inclusion $Sp_{2n} \to GL_{2n}$ induces an isomorphism $$CH^{*}(BSp_{2n}) \cong \mathbb{Z}[c_{2}, c_{4}, \ldots, c_{2n}].$$ From \cite[Corollary 6.2]{Pri}, there are isomorphisms $$H^{*}(BGL_{2n \, (1)}, \mathbb{F}_{p}(*)) \cong H^{*}(GL_{2n}, \mathbb{F}_{p}(*)) \otimes CH^{*}(BGL_{2n})/p$$ and $$H^{*}(BSp_{2n \, (1)}, \mathbb{F}_{p}(*)) \cong H^{*}(Sp_{2n}, \mathbb{F}_{p}(*)) \otimes CH^{*}(BSp_{2n})/p$$ which are induced by the fibration $$G \cong G/G_{(1)} \to BG_{(1)} \to BG$$ (for any group scheme $G/k$). Under these isomorphisms, there are generators $$g_{1}, \ldots,g_{2n} \in H^{*}(GL_{2n}, \mathbb{F}_{p}(*)) \subset H^{*}(BGL_{2n \, (1)}, \mathbb{F}_{p}(*))$$ such that $0 \neq \delta(g_{i})=p^{i-1}c_{i} \in H^{2i}(BGL_{2n \, (1)}, \mathbb{Z}(i))$ where $\delta$ is the integral Bockstein homomorphism. There is a similar description of $$H^{*}(Sp_{2n}, \mathbb{F}_{p}(*)) \subset H^{*}(BSp_{2n \, (1)}, \mathbb{F}_{p}(*)).$$ It follows that the images of the classes $g_{2}, g_{4}, \ldots, g_{2n}$ under the pullback map $$H^{*}(BGL_{2n \, (1)}, \mathbb{F}_{p}(*)) \to H^{*}(BSp_{2n \, (1)}, \mathbb{F}_{p}(*))$$ generate a copy of $H^{*}(Sp_{2n}, \mathbb{F}_{p}(*)).$ This proves the first part of the proposition.
 
\indent Proving the second part is similar, using that $$CH^{*}(BSO_{2n+1})/p \cong \mathbb{F}_{p}[c_{2}, c_{4}, \ldots, c_{2n}]$$ under the inclusion $SO_{2n+1} \hookrightarrow GL_{2n+1}$ \cite[\S 16]{TotChow}.
\end{proof}

Proposition \ref{actionsteenrod} now gives a description of the action of the Steenrod operations $P^{n}$ on $$\mathbb{F}_{p}\cdot c_{1,2}(\alpha) \oplus \mathbb{F}_{p}\cdot c_{1,4}(\alpha)\oplus \cdots \oplus \mathbb{F}_{p}\cdot c_{1,2n}(\alpha) \subset H^{*}(Sp_{2n}, \mathbb{F}_{p}(*)), H^{*}(SO_{2n+1}, \mathbb{F}_{p}(*)).$$

\indent Now we determine the motivic cohomology of the varieties $$GL_{n}/GL_{r}, Sp_{2n}/Sp_{2n-2}, SO_{2n+1}/SO_{2n-1}$$ along with the action of the Steenrod operations. For the sake of brevity, we omit the $H^{1}(k, \mathbb{F}_{p}(1))$ subgroup from   $$\bigoplus_{i=1}^{\infty}H^{2i-1}(X, \mathbb{F}_{p}(i))$$ for the varieties $X$ being considered.
\begin{proposition} \label{propcohhomospace}
Let $n, r \in \mathbb{N}.$
\begin{enumerate}
\item The quotient map $GL_{n} \to GL_{n}/GL_{r}$ induces an inclusion $$\bigoplus_{i=1}^{\infty}H^{2i-1}(GL_{n}/GL_{r}, \mathbb{F}_{p}(i)) = \mathbb{F}_{p}\cdot c_{1,r+1}(\alpha) \oplus \cdots \oplus \mathbb{F}_{p}\cdot c_{1,n}(\alpha) \subset H^{*}(GL_{n}, \mathbb{F}_{p}(*)).$$

\item The quotient map $Sp_{2n} \to Sp_{2n}/Sp_{2n-2}$ induces an inclusion $$\bigoplus_{i=1}^{\infty}H^{2i-1}(Sp_{2n}/Sp_{2n-2}, \mathbb{F}_{p}(i)) = \mathbb{F}_{p}\cdot c_{1,2n}(\alpha)$$
$$  \subset \bigoplus_{i=1}^{\infty}H^{2i-1}(Sp_{2n}, \mathbb{F}_{p}(i))=\mathbb{F}_{p}\cdot c_{1,2}(\alpha) \oplus \mathbb{F}_{p}\cdot c_{1,4}(\alpha)\oplus \cdots \oplus \mathbb{F}_{p}\cdot c_{1,2n}(\alpha).$$

\item Assume that $p>2.$ The quotient map $SO_{2n+1} \to SO_{2n+1}/SO_{2n-1}$ induces an inclusion $$\bigoplus_{i=1}^{\infty}H^{2i-1}(SO_{2n+1}/SO_{2n-1}, \mathbb{F}_{p}(i)) = \mathbb{F}_{p}\cdot c_{1,2n}(\alpha)$$
$$  \subset \bigoplus_{i=1}^{\infty}H^{2i-1}(SO_{2n+1}, \mathbb{F}_{p}(i))=\mathbb{F}_{p}\cdot c_{1,2}(\alpha) \oplus \mathbb{F}_{p}\cdot c_{1,4}(\alpha)\oplus \cdots \oplus \mathbb{F}_{p}\cdot c_{1,2n}(\alpha).$$
\end{enumerate}
\end{proposition}
\begin{proof}
There is a map of fibrations 
\begin{equation} \label{fibrations}
\begin{tikzcd}
GL_{n} \arrow[r, ""] \arrow[d, ""]& GL_{n}/GL_{r} \arrow[d, ""] \\
* \arrow[r, ""] \arrow[d, ""] & BGL_{r} \arrow[d, ""] \\
BGL_{n} \arrow[r, "id."] & BGL_{n}
\end{tikzcd}
\end{equation}
where the horizontal maps are all flat. (Here, $*$ is a really a certain $GL_{n}$-variety $U$ such that $U/GL_{n}$ approximates $BGL_{n}.$) Let $$E_{2}^{l,q}=\textup{Tor}^{CH^{*}(BGL_{n})/p}_{-l,q,j}(\mathbb{F}_{p}, H^{*}(k, \mathbb{F}_{p}(*))) \Rightarrow H^{l+q}(GL_{n}, \mathbb{F}_{p}(j))$$ and $$F_{2}^{l,q}=\textup{Tor}^{CH^{*}(BGL_{n})/p}_{-l,q,j}(\mathbb{F}_{p}, H^{*}(BGL_{r}, \mathbb{F}_{p}(*))) \Rightarrow H^{l+q}(GL_{n}/GL_{r}, \mathbb{F}_{p}(j))$$ denote the Eilenberg-Moore spectral sequences corresponding to the fibrations in \ref{fibrations}. There is a map of spectral sequences $f: F \to E$ that is compatible with pullback $$H^{*}(GL_{n}/GL_{r}, \mathbb{F}_{p}(*)) \to H^{*}(GL_{n}, \mathbb{F}_{p}(*)) .$$ The map $f$ can be defined at the level of cycles, using flat pullback.

\indent We use the Koszul resolution \cite[Corollary 4.5.5]{Wei} to compute the Tor groups in these spectral sequences. For the spectral sequence $E,$ the Tor groups are the homology of the complex $$H^{*}(k, \mathbb{F}_{p}(*)) \otimes \bigwedge(c_{1}, \ldots, c_{n})$$ where the differential $d$ is $0$ on all terms $dc_{i_{1}} \wedge \cdots \wedge dc_{i_{m}}.$ In particular, $$\bigoplus_{i=1}^{\infty}H^{2i-1}(GL_{n}, \mathbb{F}_{p}(i)) \cong \mathbb{F}_{p}\cdot dc_{1} \oplus \mathbb{F}_{p}\cdot dc_{2} \oplus \cdots \oplus \mathbb{F}_{p}\cdot dc_{n}.$$ 

\indent Similarly, for the spectral sequence $F,$ the Tor groups are the homology of the complex $$\mathbb{F}_{p}[c_{1}, \ldots, c_{n}]/(c_{r+1}, \ldots , c_{n}) \otimes H^{*}(k, \mathbb{F}_{p}(*)) \otimes \bigwedge(c_{1}, \ldots, c_{n}).$$ This implies that $$\bigoplus_{i=1}^{\infty}H^{2i-1}(GL_{n}/GL_{r}, \mathbb{F}_{p}(i)) \cong \mathbb{F}_{p}\cdot dc_{r+1} \oplus \cdots \oplus \mathbb{F}_{p}\cdot dc_{n}$$ and proves the first part of the proposition.

\indent The proofs of the second and third part of the proposition are similar, using that $$H^{*}(BSp_{2n},\mathbb{F}_{p}(*)), H^{*}(BSO_{2n+1},\mathbb{F}_{p}(*)) = H^{*}(k, \mathbb{F}_{p}(*))[c_{2}, c_{4}, \ldots, c_{2n}].$$
\end{proof}

\begin{example} Assume that $\textup{char}(k)=3.$ Consider the fibration \begin{equation} \label{fibrationSp_4}Sp_{4} \to Sp_{4}/Sp_{2}. \end{equation} From Proposition \ref{propcohhomospace}, $$\bigoplus_{i=1}^{\infty}H^{2i-1}(Sp_{4}, \mathbb{F}_{3}(i))=\mathbb{F}_{3}\cdot c_{1,2}(\alpha) \oplus \mathbb{F}_{3}\cdot c_{1,4}(\alpha)$$ and $$\bigoplus_{i=1}^{\infty}H^{2i-1}(Sp_{4}/Sp_{2}, \mathbb{F}_{3}(i))=\mathbb{F}_{3}\cdot c_{1,4}(\alpha).$$ From Proposition \ref{actionsteenrod}, $P^{1}(c_{1,2}(\alpha))=c_{1,4}(\alpha).$ This implies that the quotient map \ref{fibrationSp_4} doesn't admit a section.
\end{example}

\end{document}